\documentclass[11pt]{amsart}
\usepackage{amssymb}
\usepackage{times}
\usepackage{url}
\usepackage{hyperref}
\usepackage{xcolor}

\newtheorem{theorem}{Theorem}[section]
\newtheorem{lemma}[theorem]{Lemma}
\newtheorem{proposition}[theorem]{Proposition}
\newtheorem{corollary}[theorem]{Corollary}

\theoremstyle{definition}

\theoremstyle{remark}
\newtheorem{remark}[theorem]{Remark}

\numberwithin{equation}{section}

\newcommand{\F}{\mathbb{F}}
\newcommand{\Z}{\mathbb{Z}}
\newcommand{\Q}{\mathbb{Q}}
\newcommand{\C}{\mathbb{C}}

\newcommand{\G}{\Gamma_n}

\newcommand{\GL}{\mathrm{GL}}
\newcommand{\SL}{\mathrm{SL}}
\newcommand{\PSL}{\mathrm{PSL}}
\newcommand{\Sp}{\mathrm{Sp}}
\newcommand{\PSp}{\mathrm{PSp}}

\newcommand{\calg}{{\mathcal G}}

\newcommand{\cha}{\mathrm{char} \hspace{.5pt}}

\newcommand{\abk}{\allowbreak}

\begin{document}

\title[Zariski density and computing in arithmetic groups]{Zariski
density and computing in arithmetic groups}

\author{A.~Detinko}

\address{School of Computer Science\\
University of St~Andrews\\
North Haugh\\
St~Andrews\\
KY16 9SX\\
UK}

\author{D.~L.~Flannery}
\address{
School of Mathematics, Statistics and Applied Mathematics\\
National University of Ireland, Galway\\
University Road\\
Galway\\
Ireland}
\thanks{
The first and second authors received support from the Irish Research
Council 
(grants `Mat-GpAlg' and `MatGroups') and Science Foundation Ireland
(grant 11/\abk RFP.1/\abk MTH/3212). The first author is also funded by 
a Marie Sk\l odowska-Curie Individual Fellowship grant under Horizon 2020
(EU 
Framework Programme for Research and Innovation).}
\author{A.~Hulpke}
\address{Department of Mathematics\\
Colorado State University\\
Fort Collins, CO 80523-1874}
\thanks{
The third author was supported 
by Simons Foundation Collaboration Grant~244502}

\subjclass[2010]{20H05,20B40}

\date{}

\begin{abstract}
For $n > 2$, let $\G$ denote either $\SL(n, \Z)$ or $\Sp(n, \Z)$.
We give a practical algorithm to compute the level of the maximal 
principal congruence subgroup in an arithmetic group $H\leq \G$. 
This forms the main component of our methods for computing with 
such arithmetic groups $H$. More generally, we provide 
algorithms for computing with Zariski dense groups in $\G$. 
We use our {\sf GAP} implementation of the algorithms to solve 
problems that have emerged recently for important classes of 
linear groups.
\end{abstract}

\maketitle

\section{Introduction}

This paper is the next phase in our ongoing project to build 
up a new area of computational group theory: computing with linear 
groups given by a finite set of generating matrices over an infinite 
field. Previously we established a uniform approach for handling such 
groups in a computer. This is based on the use of congruence 
homomorphisms, taking advantage of the residual finiteness of finitely 
generated linear groups: a realisation of the `method of finite 
approximation'~\cite{Recog}. 
We verified decidability, and then obtained efficient algorithms for solving 
problems such as testing finiteness and virtual solvability. We also implemented 
a suite of algorithms to perform extensive structural investigation of 
solvable-by-finite linear groups. 

Most finitely generated linear groups, however, are not virtually 
solvable, and computing with those groups is largely unexplored
territory. Obstacles include undecidability of certain algorithmic 
problems, complexity issues (e.g., growth of matrix entries), and a 
dearth of methods. In \cite{Arithm}, we initiated the development of 
practical algorithms for arithmetic subgroups of semisimple algebraic 
groups $\mathcal G$ defined over the rational field $\Q$. We were 
motivated by the pivotal role that these groups play throughout 
algebra and its applications, and the concomitant demand for 
practical techniques and software to work with them. 

At this stage we restrict attention to $\mathcal G$ possessing
the \emph{congruence subgroup property}: each arithmetic group 
$H \leq \mathcal{G}(\Z)$ contains the kernel of the congruence 
homomorphism on $\mathcal{G}(\Z)$ modulo some positive integer $m$, 
the so-called \emph{principal congruence subgroup (PCS) of level $m$}. 
Prominent examples are $\mathcal{G}(\Z)=\SL(n, \Z)$ 
and $\Sp(n, \Z)$ for $n > 2$ (see \cite{Bass}). The congruence 
subgroup property allows us to reduce much of the computing to the 
environment of matrix groups over finite rings; but we first need 
to know (the level of) a PCS in $H$. In \cite{Arithm} we showed 
that construction of a PCS in an arithmetic group $H\leq \SL(n, \Z)$ 
is decidable. As a consequence, this proves that other algorithmic 
questions (e.g., membership testing, orbit-stabilizer problems, 
analyzing subnormal structure) are decidable, and yields algorithms 
for their solution. 

The current paper gives a practical algorithm to compute a 
PCS in an arithmetic subgroup $H\leq \G = \SL(n, \Z)$ or $\Sp(n, \Z)$ 
for degrees $n > 2$. More precisely, we compute the \emph{level} 
$M = M(H)$ of $H$, i.e., the level of the unique maximal PCS in $H$. 
Knowing $M$, we can undertake further computation with $H$ (this 
subsumes all algorithms from \cite{Arithm}). 

In contrast to computing with a virtually solvable linear group, 
computing with an arithmetic group $H\leq \G$ entails reduction 
modulo ideals that may not be maximal. Moreover, we must consider 
images of $H$ modulo all primes. Fortunately, $H$ and $\G$ are 
congruent modulo $p$ for almost all primes $p$. This property holds
in a wider class, namely subgroups of $\G$ that are dense in the 
Zariski topology on $\SL(n,\C)$, respectively $\Sp(n,\C)$. Density 
is weaker than arithmeticity, easier to test, and indeed 
furnishes a preliminary step in arithmeticity testing (see \cite{Sarnak}
for justification of the significance of this problem). 
Dense non-arithmetic subgroups are called \emph{thin} matrix groups. 
If $H$ is dense (either arithmetic or thin),
then by the strong \mbox{approximation} theorem 
$H$ surjects onto $\SL(n, p)$, respectively $\Sp(n, p)$, modulo all 
but a finite number of primes $p$~\cite[p.~391]{LubotzkySegal}. We 
design effective algorithms to compute the set $\Pi(H)$ of these
primes for finitely generated $H \leq \G$ containing a 
transvection. As a by-product, we get a simple algorithm to test 
density of such groups (albeit for odd $n$ only if $\G=\SL(n,\Z)$). 
Computing $\Pi(H)$ when $H$ does not have a known transvection will 
be dealt with in a subsequent paper~\cite{DensityFurther}.
Our next major result shows that the 
algorithm to compute the level of the maximal PCS of an arithmetic 
subgroup also finds the minimal arithmetic overgroup 
$L$ of a finitely generated dense subgroup $H$ of $\G$. Algorithms for 
the arithmetic group $L$ (e.g., as in \cite{Arithm}) 
can thereby be used to study $H$.

When computing with arithmetic groups, the relevant congruence images 
are over finite rings $\Z_m := \Z/ m\Z$ (for virtually solvable groups,
the images are over finite fields). We prove some essential results 
about subgroups of $\GL(n,\Z_m)$ in Subsection~\ref{Subsection12}. These 
underlie Subsection~\ref{Subsection13}, wherein we present our algorithm 
to compute the level $M$ of an arithmetic group in $\G$. 
Section~\ref{Section2} is dedicated to density testing and computing 
$\Pi(H)$ for a finitely generated dense group $H\leq \G$. In 
Section~\ref{Section3}, we use our algorithms to solve  
computational problems that have recently emerged for important 
classes of groups. The experimental results demonstrate the efficiency 
of our algorithms. Finally, in Section~\ref{Miscellaneous} 
we discuss our {\sf GAP}~\cite{Gap} implementation of density testing 
algorithms, including those from \cite{Rivin1}.

\section{The level of an arithmetic subgroup}\label{Section1}

In this section we develop techniques for computing the level of
an arithmetic group in $\Gamma_n$.
\subsection{Setup}
We adhere to the following notation. 
Let $R$ be a commutative unital ring. The symplectic group of 
degree $n=2s$ over $R$ is
\[
\Sp(n,R)=\{ x\in \GL(n,R) \ | \ x J x^\top = J\}
\]
where 
\[
J={
\left(\renewcommand{\arraycolsep}{.17cm}\! 
\begin{array}{rr} 0_{s} & 1_{s}
\\
\vspace*{-9pt}& \\
-1_{s} & 0_{s}
\end{array} \! \right)}. 
\]
Notice that $\Sp(2,R) = \SL(2,R)$.
Let $t_{ij}(m)=1_n + \abk me_{ij}\in \SL(n,R)$,
where $e_{ij}$ has $1$ in position $(i,j)$ and zeros elsewhere.
Define
\[
E_{n,m} = \langle \hspace{.1mm} t_{ij}(m)\, :  \, i\neq j, \, 1 \leq
i, j \leq n \hspace{.1mm} \rangle
\]
if $\G = \SL(n, R)$, and
\begin{align*}
E_{n,m} = & \ \,  \{ t_{i,s+j} (m)\hspace{.75pt}
t_{j,s+i}(m), \hspace{.25pt} t_{s+i,j}(m)\hspace{1pt} t_{s+j,i}(m)
\; | \; 1\leq i < j\leq
s\}\\
& \ \ \cup \hspace{.5pt} \{ t_{i,s+i}(m), t_{s+i,i}(m) \; | \;
1\leq i\leq s\}
\end{align*}
if $\G = \Sp(2s, R)$. The $E_{n,m}$ are
\textit{elementary subgroups of $\, \G$ of level $m$}
(\cite[Section~1.1]{Arithm}, \cite[pp.~223--224]{HahnOMeara}). 
For $R=\abk \Z$ or $\Z_r$ we have $E_{n,1}=\SL(n,R)$ if 
$\Gamma_n = \SL(n,R)$ and $E_{n,1}=\Sp(n,R)$ if 
$\Gamma_n = \Sp(n,R)$. 

The reduction modulo $m$ map $\varphi_m$
on $R = \Z$ or $R= \Z_r$ 
extends entrywise to a homomorphism on 
$\GL(n,R)$, also denoted by $\varphi_m$.
This congruence homomorphism maps $\Gamma_n$ 
onto $\SL(n,\Z_m)$ or $\Sp(n,\Z_m)$ respectively.
For $\G= \SL(n,\Z)$ and $n> 2$, the normal closure 
$E_{n,m}^{\hspace{.2pt} \G}$ is the principal congruence subgroup 
(PCS) of level $m$, i.e., the kernel 
of $\varphi_m$ on $\Gamma_n$, denoted
$\Gamma_{n,m}$~\cite[Proposition~1.6]{Arithm}. Similarly,
$E_{n,m}^{\hspace{.2pt} \G}$ is the kernel
$\Gamma_{n,m}$ of $\varphi_m$ on $\G= \abk \Sp(n,\Z)$
when $n> 2$~\cite[Proposition~13.2]{Bass}. Let $H\leq \Gamma_n$.
As usual $\Pi := \Pi(H)$ is the set of primes $p$ 
such that $\varphi_p(H)\neq\varphi_p(\Gamma_n)$.
If $|\Gamma_n:H|$ is finite then $H$ contains some 
$\Gamma_{n,m}$~\cite{Bass}. Indeed, $H$ contains a unique maximal 
PCS; its level is defined to be the level $M = M(H)$ of $H$. 

\subsection{Decidability}

Let $n> 2$.
Decision problems for arithmetic groups $H$ in $\G = \abk \SL(n,\Z)$ were 
discussed in \cite[Section~3.1]{Arithm}. Here we cover $\G = \Sp(n,\Z)$
as well. 
\begin{lemma}\label{Lemma11}
Suppose that 
$\Gamma_{n,m}\leq H$.
If $k$ is coprime to $m$,
 then $\varphi_k(H) =\abk \varphi_k (\G)$.
\end{lemma}
\begin{proof}
Cf.~\cite[Remark~1.18]{Arithm}.
\end{proof}

Denote the set of prime divisors of $a \in \Z$ by $\pi(a)$.
\begin{corollary}\label{Pivspi}
If $H$ has level $M$ and $\Gamma_{n,m}\leq H$
then $\Pi \subseteq \pi(M) \subseteq \pi(m)$. 
\end{corollary}

\begin{proposition}\label{Decidable}
Let $H\leq \G$ be arithmetic. Then computing the level of a PCS in $H$
is decidable.
\end{proposition}
\begin{proof}
We can compute $c=|\G:H|$ by the Todd--Coxeter procedure. 
The core $H_{\G}$ is a normal subgroup of $\G$ contained in $H$, and
$|\G:H_{\G}|$ divides $m:= \abk c!$. So $x^m\in H_{\G}$ for all $x\in \G$.
Thus $E_{n,m}\leq H_{\G}$, and $\Gamma_{n,m}\leq H$.
\end{proof}
\begin{corollary}
\label{corollary15double}
If $H\leq \G$ is arithmetic then
testing membership of $g \in \G$ in $H$ is decidable.
\end{corollary}

Knowing the level $m$ of any PCS in $H$, we can determine $|\G:H|$
and the level $M$ of $H$.
Therefore, computing $M$ is equivalent to computing
$|\G:H|$. According to \cite[pp.~531--532]{GSI}, an arithmetic subgroup 
$H\leq \calg({\Z})$ of an algebraic $\Q$-group $\calg \leq \GL(n,\C)$ 
is \textit{given explicitly} if there is an effective way 
to test membership of elements of $\calg({\Z})$ in $H$, and we have 
an upper bound on $|\calg({\Z}) : H|$. 
By Proposition~\ref{Decidable} and  Corollary~\ref{corollary15double}, 
every arithmetic subgroup of $\calg = \SL(n,\C)$ or $\Sp(n,\C)$ is 
explicitly given. This guarantees decidability of 
key algorithmic 
problems for these groups, as in \cite{GSI} (see also 
\cite[Section~3.1]{Arithm}).

In practice, we would not compute $M$ or $|\G:H|$ as in the 
proof of Proposition~\ref{Decidable}: the runtime of the
Todd--Coxeter procedure may be arbitrarily large. 
Subsection~\ref{Subsection13} gives a practical method for
computing $M$. This requires extra results, to be 
presented in the next subsection. 

\subsection{Existence of supplements of congruence
subgroups over $\Z_m$}
\label{Subsection12}

For specified $m$, denote the kernel of 
the reduction modulo $r$ homomorphism 
$\varphi_{r}$ on 
$\varphi_{m}(\Gamma_n)=\SL(n,\Z_m)$ or 
$\Sp(n,\Z_m)$ by $K_{r}$.
Let $p$ be a prime. Our main objective in this subsection is to prove
the following theorem.
\begin{theorem}
\label{NonSplitLinear}
Let $a$, $n\ge 2$, and $G=\SL(n,\Z_{p^a})$ or $\Sp(2n,\Z_{p^a})$.
Then $K_{p^{a-1}}$ has a proper supplement in $G$ if and only if
$G$ is one of
\begin{equation}\label{Exceptions}
\SL(2,\Z_{4}), \quad \SL(2,\Z_{9}), \quad
\SL(3,\Z_{4}), \quad \SL(4,\Z_{4}).
\end{equation}
\end{theorem}
Part of the symplectic group case of Theorem~\ref{NonSplitLinear} 
is treated in~\cite{hulpkespgens}.

We need several preparatory lemmas.
\begin{lemma}
\label{KForm} Let $a\geq 2$.
\begin{itemize}
\item[{\rm (i)}] The kernel $K_{p^{a-1}}$ of
$\varphi_{p^{a-1}}\colon \SL(n,\Z_{p^a})\to\SL(n,\Z_{p^{a-1}})$ is
\begin{equation*}
\left\{1_n+p^{a-1}x\; \big| \; x\in \mathrm{Mat}(n,\{0,1,\ldots, p-1\}),
\, \mathrm{trace}(x) \equiv 0 \ \mathrm{mod}\  p \right\}.
\end{equation*}
\item[{\rm (ii)}] Multiplication in $K_{p^{a-1}}$ translates to matrix
addition in $\mathrm{Mat}(n,\Z_{p})\hspace{-2pt} :$
\[
(1_n+p^{a-1}x)(1_n+p^{a-1}y) = 1_n+p^{a-1}z
\]
where $z\equiv x+y$ mod $p$. In particular, $K_{p^{a-1}}$ is an
elementary abelian $p$-group.
\end{itemize}
\end{lemma}
\begin{proof}
By induction on $n$,
$\mathrm{det} (1_n+p^{a-1}x) = 1 + p^{a-1}\mathrm{tr}(x)$,
so (i) follows. The other part is obvious.
\end{proof}

\begin{lemma}
\label{lemmanosupplement} 
Let $a\ge 3$ and $G=\SL(n,\Z_{p^a})$
for $n\geq 2$ or $G = \Sp(n,\Z_{p^a})$
 for $n\geq 4$. Then
$K=K_{p^{a-1}}\leq G$ is a central subgroup of 
 $L=K_{p^{a-2}}\leq G$
and has no proper supplement in $L$.
\end{lemma}
\begin{proof}
Since $2a-3\ge a$,
\begin{align*}
(1_n+p^{a-1}x)(1_n+p^{a-2}v)&=1_n+p^{a-1}x+p^{a-2}v\\
&= (1_n+p^{a-2}v)(1_n+p^{a-1}x)
\end{align*}
in $\mathrm{Mat}(n,\Z_{p^{a}})$.
Thus $K\leq Z(L)$.

The subgroup $K$ is generated by $p$th powers of elements of $L$.
If $L=KU$ then $K = L^p = K^p U^p\le U$; hence $U=L$.
\end{proof}

\begin{lemma}
\label{lemmabuildup}
Let $G=\Sp(4,\Z_{p^2})$ and $H=\Sp(2,\Z_{p^2})$, $p$ odd.
Denote by $C$, $D$ the kernel of $\varphi_p$ on $G$, $H$,
respectively. If there is a proper supplement of $C$ in $G$
then there is a proper supplement of $D$ in $H$.
\end{lemma}
\begin{proof}
The assignment
\[
\lambda:
{
\left(\begin{array}{cc}a&b\\
c&d\\
\end{array}\right)
\mapsto\left(\begin{array}{cccc}
a & 0 & b & 0\\
0 & 1 & 0 & 0\\
c & 0 & d & 0 \\
0 & 0 & 0 & 1
\end{array}\right)}
\]
defines an embedding $\lambda\colon H \to G$.
Clearly $\lambda(D)\le C$. Let $N$ be the subgroup
of $C$ whose elements are of the form $1_4+pr$ where
\begin{equation}
\label{symkmat2}
r={
\left(\! \!
\begin{array}{rcrr}
0    & \ \ v_1 &   0   & w_1 \\
-w_3 & \ \ v_2 &   w_1 & w_2 \\
0    & \ \ v_3 &   0   & w_3 \\
v_3  & \ \ v_4 &  -v_1 & -v_2
\end{array}
\right)}.
\end{equation}
Then $N$ complements $\lambda(D)$ in $C$, and is normalized 
by $\lambda(H)$. Therefore $N$ is a normal subgroup of 
$W=C \lambda(H)= N\lambda(H)$.

The natural epimorphism $\kappa\colon W\to H$ with kernel
$N$ maps $C$ to $D$.
Suppose that $S$ is a supplement of $C$ in $G$.
Then $\kappa(S\cap W)$ supplements $D$ in $H$; if it does
not do so properly then $\kappa(S\cap C) =D$.
Assuming this, we prove that $C\le S$, i.e., $S=G$.

Note that $S\cap C \unlhd G$. As Lemma~\ref{KForm}
indicates, we may view $C$ as an additive subspace of
$\mathrm{Mat}(4,\Z_{p})$, replacing $1_4+px$
by its {\em relic} $x$ with entries
in $\{0,1,\abk \ldots, \abk p-1\}$. Since $S\cap C$ surjects
onto $D$, it must contain a relic $a=r+\abk e_{31}$
with $r$ as in \eqref{symkmat2}. Let $b_1=1_4+e_{13}\in G$
and $a_1=\abk a-\abk a^{b_1}=\abk
e_{11}+e_{13}-e_{33}+v_{3}(e_{12}-e_{43})+
w_{3}(e_{23}+e_{14})$.
Then $a_2=\frac{1}{2}(a_1^{b_1}-a_1)=e_{13}\in S\cap C$.

Since $G$ contains a permutation matrix $t$
for $(1,2)(3,4)$, $e_{24}=e_{13}^t \in S\cap C$.
We construct more elements in $S\cap C$.
Let $d_1=1_4+e_{31}$, $d_2=1_4+e_{14}+e_{23}$, and
$d_3=1_4+e_{32}+e_{41}$. Then
\begin{itemize}
\item[] $\frac{1}{2}(a_2^{d_1}-a_2^{d_1^{-1}})=e_{11}-e_{33}$;
conjugating by $t$ gives $e_{22}-e_{44}$,

\item[] $\frac{1}{2}(a_2^{d_3}-a_2^{d_3^{-1}})=e_{12}-e_{43}$;
conjugating by $t$ gives $e_{21}-e_{34}$,

\item[] $\frac{1}{2}(-a_2^{d_1}+a_2^{d_1^{-1}}+
a_2^{d_1d_2}-a_2^{d_1^{-1}d_2})=e_{14}+e_{23}$,

\item[] $\frac{1}{2}(2a_2-a_2^{d_1}-a_2^{d_1^{-1}})=e_{31}$;
conjugating by $t$ gives $e_{42}$,

\item[] $-a_2+a_2^{d_1}+a_2^{d_3}-a_2^{d_1d_3}=e_{32}+e_{41}$.
\end{itemize}
Modulo $p$, the relics of $C$ are exactly those matrices $x$ such 
that $Jx$ is symmetric. Thus $C$ has $\Z_p$-dimension $10$.
It is readily checked that $e_{13}$, $e_{24}$, and the eight other 
elements of $S\cap C$ just listed are linearly independent.
So they comprise a basis of $C$. Therefore $C\le S$ as claimed.
\end{proof}

Everything is in place to prove Theorem~\ref{NonSplitLinear}.
\begin{proof}
If there were a proper supplement of $K_{p^{a-1}}$ in
$G$ then there would be a proper supplement of $K_{p^{a-1}}$
in $K_{p^{a-2}}$. So fix $a=2$ by 
Lemma~\ref{lemmanosupplement}.

We appeal to \cite{Beisiegel}, \cite[Theorem~1]{Weigel1},
and \cite{Weigel2}. If $G$ is not one of the groups in \eqref{Exceptions}, 
or if $G=\Sp(n,\Z_{p^2})$ for $n\geq 6$ but $G\neq \Sp(6, \Z_{4})$, then 
these results imply that $K_p$ lies in the Frattini
subgroup of $G$. Therefore $G$ does not have a proper supplement. 

A standard {\sf GAP} computation reveals that if $G$ is one of 
the groups in \eqref{Exceptions} then $K_p$ is properly supplemented.

Since $\Sp(2,R) = \SL(2,R)$, it remains to prove non-supplementation 
of $K_p$ in $\Sp(6,\Z_{4})$ and $\Sp(4,\Z_{p^2})$. The latter  
follows from Lemma~\ref{lemmabuildup} for $p>3$; the other parts 
may be verified by {\sf GAP} computations.
\end{proof}

\subsection{Computing the level}
\label{Subsection13}

We now develop an algorithm to compute the level $M$
of an arithmetic group $H\leq \G$, provided that the set 
$\pi(M)$ of primes dividing $M$ is known. To fulfil this requirement,
we determine the precise relationship between $\pi(M)$ and $\Pi(H)$.

\subsubsection{}
By Corollary~\ref{Pivspi}, $\Pi(H)\subseteq \pi(M)$. We will prove 
conversely that the odd part of $\pi(M)$ coincides with $\Pi(H)$. 
Furthermore, we show how to decide whether $M$ is even.

Below, $\Sigma$ stands for either $\SL$ or $\Sp$ 
(if $\Sigma = \Sp$ then of course the degree is even).
\begin{lemma}
\label{solradlemma}
Let $n\ge 3$, $p$ be a prime, $a\geq 1$, and $G=\Sigma(n,\Z_{p^a})$. 
Every proper normal subgroup of $G$ lies in the solvable radical 
$R$ of $G$, unless $G=$ $\Sp(4,\Z_{2^a})$, which has a subgroup of 
index $2$ containing $R= \ker \, \varphi_2$ (the only proper normal 
subgroup of $G$ not in $R$).
\end{lemma}
\begin{proof}
It is easily seen that $R$ is the full preimage in $G$ of the center 
of $\Sigma(n,p)$ under $\varphi_p$. If $G=\Sp(4,\Z_{2^a})$ then 
$G/R\cong \mathrm{Sym}(6)$. Excluding this case, $G/R$ is simple. 
By Theorem~\ref{NonSplitLinear} and induction, if 
$N\not \subseteq R$ is a proper normal subgroup of $G$ then $n=3$ or 
$4$ and $p=a=2$. 
{\sf GAP} computations confirm that such an $N$ does not exist in 
$\SL(3,\Z_4)$ or $\SL(4,\Z_4)$. The remainder of the proof is consigned
as a straightforward exercise. 
\end{proof}

We thank Derek Holt for sharing the next lemma and its proof with us. 
Recall that a {\em section} of a group $G$ is a quotient of 
a subgroup of $G$.
\begin{lemma} {\rm (D.~F.~Holt.)}
\label{sectionthensub}
Let $S$ be a finite nonabelian simple group that is not a section of 
$\PSp(4,2)$, i.e., $S$ is not isomorphic to $\mathrm{Alt}(5)$ or 
$\mathrm{Alt}(6)$. Suppose that $S$ is a section of a finite classical 
group $G$ of degree $n$ in characteristic $p$.

Then there exists a finite classical group $\hat G$ of degree
$n$ in characteristic $p$, or of degree less than $n$, such that 
a subgroup of $\hat G/Z(\hat G)$ is isomorphic to $S$. 
\end{lemma}
\begin{proof}
Suppose that $G$ is a counterexample with $n$ minimal and $|G|$ minimal 
for this $n$, with $S$ a quotient of $H\le G$ and $|H|$ minimal as well. 
Since $S$ is simple, it must be a section of $G/Z(G)$. We may therefore 
assume that $Z(G)\le H$. When $H=G$, the non-solvable group $G/Z(G)$ 
is either simple (then the lemma holds for 
$\hat G=G$); or, if it is $O^+_4(q)$, a direct product of two 
copies of $\PSL(2,q)$ (then we can take $\hat G=\PSL(2,q)$). 
So suppose that $H\not=G$. We apply Aschbacher's theorem~\cite{Aschbacher84} 
to $H$.

If $H$ is in class $\mathrm{C}_1$, $\mathrm{C}_3$, or $\mathrm{C}_4$, 
then $S$ is a section of a classical group of smaller degree in 
characteristic $p$, contrary to the minimality of $n$.
If $H\in \mathrm{C}_2\cup \mathrm{C}_7$ then either the same is true 
or $S$ is a section of $\mathrm{Sym}(n)$ with $n\ge 5$. However, 
$\mathrm{Sym}(n)$ has a faithful representation of degree $n-1$ (in any 
characteristic).

If $H\in \mathrm{C}_5\cup \mathrm{C}_8$ then $S$ is a 
section of a smaller classical group of degree 
$n$ in characteristic $p$, contradicting minimality of $|G|$.

If $H\in \mathrm{C}_6$ then $n=r^k>2$ is an odd prime power and 
$P$ is a section of $\PSp(2k,r)$. Unless $k=r=2$ we have $2k<n$.
In the remaining case $k=\abk r=\abk 2$, $S$ would have to be a 
section of $\PSp(4,2)$; which was ruled out from the beginning.
\end{proof}

\begin{lemma}
\label{noncommfactor}
Let $n\ge 3$ and $p$, $q$ be distinct primes.
If $P=\mathrm{P}\Sigma(n,q)$ is a section of $\, \Sigma(n,p)$,
then $P=\abk \mathrm{PSL}(3,2)$ or $\mathrm{PSp}(4,2)$.
\end{lemma}
\begin{proof}
Suppose that $P=\mathrm{P}\Sigma(n,q)$ is a section of the classical group
$G=\Sigma(n,p)$. By Lemma~\ref{sectionthensub}, $P$ will be
isomorphic to a subgroup of $\hat G/Z(\hat G)$ for a classical group $\hat
G$ of degree less than $n$, or of degree $n$ and in the same characteristic 
$p$ as $G$. This implies that $P$ has a faithful projective representation 
$\rho$ of degree less than $n$, or of degree $n$ in characteristic $p\not= q$.

Let $\Sigma = \SL$. By \cite[p.~419]{landazuriseitz}, the smallest
degree of a non-trivial projective representation of $\PSL(n,q)$ 
in characteristic $p\not=q$ and for $(n,q)\neq (3,2)$ is at least 
$q^{n-1}-1>n$. In characteristic $q$, the minimal degree is 
$n$ by \cite[p.~201]{kleidmanliebeck}. So the existence of $\rho$ 
disposes of this option.

Now let $\Sigma = \Sp$.
The same references as above give the smallest degree $d_p(n,q)$ of
a faithful projective representation  of $\PSp(n,q)$ in characteristic $q$ as
$d_q(n,p)=n$, in coprime characteristic $p\not=q$ as
$d_p(n,q)=\frac{1}{2}(q^{\frac{n}{2}}-1)$ for odd
$q$, $d_p(n,2)=2^{\frac{n}{2}-2}(2^{\frac{n}{2}-1}-1)$ for $n>6$, and
$d_p(6,2)=7>6$. Unless $(n,q)=(4,3)$, the existence of $\rho$ once more gives 
a contradiction.

The only remaining case is $P=\PSp(4,3)$ as
a section of $\PSp(4,q)$ for $q\not=3$.  Inspection of the maximal subgroups
of $\PSp(4,q)$ (see \cite[Tables~8.12 and 8.13]{braydougalholt}) shows that 
this is impossible.
\end{proof}

By happenstance the exceptions of Lemma~\ref{noncommfactor}
are close to those of Theorem~\ref{NonSplitLinear} and 
Lemma~\ref{solradlemma} (degree at most $4$ and characteristic
a power of $2$). So the prime $p=2$ will be treated as exceptional if the 
degree $n$ is $3$ or $4$. If $n>4$ or $p$ is an odd prime then we call 
the pair $(n,p)$ {\em unexceptional}.
\begin{lemma}
\label{kernelcor}
Let $n\ge 3$, $a\ge 1$, $q>1$, and $p$ be a prime not dividing $q$
such that $(n,p)$ is unexceptional.
Suppose that $U\le\Sigma(n,\Z_{p^aq})=\Sigma(n,\Z_{p^a})\times\Sigma(n,\Z_q)$ 
maps onto $\Sigma(n,\Z_{p^a})$ under natural projection of the
entire direct product onto its first factor. 
Then $U$ contains $\Sigma(n,\Z_{p^a})$.
\end{lemma}
\begin{proof}
The group $U$ is a subdirect product of its projections $A$, $B$ into 
$\Sigma(n,\Z_{p^a})$ and $\Sigma(n,\Z_q)$, respectively, where
$A$ is all of $\Sigma(n,\Z_{p^a})$. Assuming first that $q=\abk r^b$ is 
a prime power, we claim that $B\leq U$, i.e., $U=A\times B$. If not, then 
$A$ and $B$ have isomorphic non-trivial quotients. 
By Lemma~\ref{solradlemma}, any non-trivial quotient of $A$ has 
a quotient isomorphic to $A/R\cong {\rm P}\Sigma(n,p)$, where 
$R$ is the solvable radical of $A$. In turn, $A/R$ must be a section 
of the radical quotient ${\rm P}\Sigma(n,r)$ of $\Sigma(n,\Z_q)$: 
but this contradicts Lemma~\ref{noncommfactor}.

Suppose now that $q=r^b s$ with $r$ prime and $\gcd(r,s)=1$.
By the preceding paragraph, $\Sigma(n,\Z_{p^a})\leq \varphi_{p^ar^b}(U)$.
Thus $U\cap \Sigma(n,\Z_{p^as})$ projects onto $\Sigma(n,\Z_{p^a})$ 
modulo $p^a$. The lemma follows by induction.
\end{proof}

\begin{theorem}\label{SurjectiveImpliesCoprime}
Let $n> 2$ and $H\leq \Gamma_n$ be arithmetic of 
level $M>1$. Suppose that $(n,p)$ is unexceptional 
and $\varphi_p(H)=\Sigma(n,p)$. Then $p\nmid M$.
\end{theorem}
\begin{proof}
Assume that $\varphi_{p^{k-1}}(H)=\Sigma(n,\Z_{p^{k-1}})$ for some
$k\geq 2$. 
Then $\varphi_{p^k}(H)$ is a supplement of $\ker \varphi_{p^{k-1}}$ in
$\Sigma(n,\Z_{p^k})$, so $\varphi_{p^k}(H)=\Sigma(n,\Z_{p^k})$ by
Theorem~\ref{NonSplitLinear}.  
Hence $\varphi_{p^k}(H)=\SL(n,\Z_{p^k})$ for all $k\geq 1$ by induction.

If $M=p^a q$ with $\mathrm{gcd}(p, q)=1$ then 
$\Gamma_{n,q}\leq H$ by Lemma~\ref{kernelcor}. 
Since $H$ has level $M$, this forces $a=0$.
\end{proof}

\begin{corollary}
Suppose that $H\leq \G$ is arithmetic, $n> 2$, and 
$\varphi_p(H) = \varphi_p(\G)$ for all odd primes $p$. Then the 
level of $H$ is a $2$-power. If additionally $n\ge 5$ and 
$\varphi_2(H)=\varphi_2(\G)$ then $H=\G$.
\end{corollary}
\begin{remark} 
There are finitely generated subgroups $H$ of $\G$ with infinite
index such that $\Pi(H) = \emptyset$; see \cite{Humphreys,VenkySoifer}.
\end{remark}

Let $H\le \Gamma_n$ for $n> 2$. Define
\[
\delta_H(m)=|\Gamma_n:\Gamma_{n,m}H|.
\]
That is, $\delta_H(m)=|\varphi_m(\Gamma_n):\varphi_m(H)|$. 
We record a few properties of the delta function.
\begin{lemma}\label{BasicDeltaProperties} 
Let $m$, $m'$ be positive integers.
\begin{itemize} 
\item[{\rm (a)}] If $m \hspace{1.4pt} | \hspace{1.4pt} m'$ then
$\delta_H(m) \hspace{1.4pt} | \hspace{1.4pt} \delta_H(m')$.
\item[{\rm (b)}] Suppose that $H$ is arithmetic of level $M$, so
$\delta_H(M) =|\Gamma_n:H|$. Then 
\begin{itemize}
\item[{\rm (i)}] 
$\delta_H(m) \hspace{1.4pt} | \hspace{1.4pt}\delta_H(M)$
\item[{\rm (ii)}]  $\delta_H(m)=\delta_H(M)$ if and only if 
$M \hspace{1.4pt} | \hspace{1.4pt} m$.
\end{itemize}
\end{itemize}
\end{lemma}

The next theorem gives a criterion to test whether $M$ is even
(when $2\not \in \Pi$).
\begin{theorem}
\label{evenMtest}
Let $n> 2$ and $H\leq \Gamma_n$ be arithmetic of 
level $M>1$. Let $q$ be the product of all odd primes in $\Pi(H)$.
Then $M$ is even  if and only if $\delta_H(q)<\delta_H(4q)$.
\end{theorem}
\begin{proof}
By Theorem~\ref{SurjectiveImpliesCoprime}, $q$ is the product of 
all odd primes dividing $M$.

If $M$ is odd then $\Gamma_{n,4q}\Gamma_{n,M}=\Gamma_{n,q}$, so
$\delta_H(4q)=\delta_H(q)$.

For the rest of the proof, suppose that $M$ is even: say $M=2^ls$, 
$l\geq 1$, $s\ge 1$ odd. By Lemma~\ref{BasicDeltaProperties}, 
$\delta_H(s)<\delta_H(2^ks)$ for $1\le k\le l$. So choose the 
least $k$ and least multiple $r$ of $q$ dividing $s$ such that 
$\delta_H(r)<\delta_H(2^k r)$. Then let $m=2^kr$, 
$A=\abk \varphi_{2^k}(H)$, $B=\abk \varphi_r(H)$, 
and $N=\varphi_m(\Gamma_{n,r}\cap H)$.

If $A\not=\Sigma(n,\Z_{2^k})$ then by the same argument 
as in the first paragraph of the proof of 
Theorem~\ref{SurjectiveImpliesCoprime} (here avoiding the 
exceptions for $p=2$ in small degrees), 
$\varphi_4(H)\not=\Sigma(n,\Z_4)$. 
So $\delta_H(q)<\delta_H(4q)$.

Henceforth $A=\Sigma(n,\Z_{2^k})$. Then $\varphi_{2^k}(N)$ is a 
proper (normal) subgroup of $A$: otherwise, 
$\Gamma_{n,r}\leq \Gamma_n = (\Gamma_{n,r}\cap H) \Gamma_{n,2^k}
\Rightarrow \Gamma_{n,r}\leq 
(\Gamma_{n,r}\cap H)(\Gamma_{n,2^k}\cap \Gamma_{n,r})
\Rightarrow \Gamma_{n,r}\leq (\Gamma_{n,r}\cap H) \Gamma_{n,m}
\Rightarrow \delta_H(r)=\delta_H(m)$. By Lemma~\ref{solradlemma}, 
$\varphi_{2^k}(N)R\neq \abk A$ where $R$ is the solvable
radical of $A$. Since $R$ contains the kernel of $\varphi_2$ 
on $A$, $\varphi_2(N) \neq \Sigma(n,2)$. We conclude that 
$\delta_H(r)<\delta_H(2r)$. 
Therefore $k=1$.

Let $L=\varphi_m(\Gamma_{n,2}\cap H)$.
Then $L\neq \varphi_m(H)$ and $A/\varphi_2(N)\cong B/\varphi_r(L)$. 
Let $K$ be the kernel of $\varphi_q$ on $\varphi_r(H)$, i.e., 
$K= \varphi_r(H)\cap \varphi_r(\Gamma_{n,q})=\varphi_r(H\cap \Gamma_{n,q})$. 
We show that $K\varphi_r(L)\not=B$. This will imply that $\varphi_{2q}(H)$
is a proper subdirect product of $A=\varphi_2(H)$ and $\varphi_q(H)$, so 
$\delta_H(q)<\abk \delta_H(2q)\leq \delta_H(4q)$ as desired.

If $A/\varphi_2(N)$ is solvable then $|A:\varphi_2(N)|=2$ by 
Lemma~\ref{solradlemma}. Thus  $K\varphi_r(L)\not=B$ because $K$ has 
odd order. If $A/\varphi_2(N)$ is not solvable then neither is 
$B/\varphi_r(L)$, and the result again follows. 
\end{proof}

For $n> 2$ and any $H\leq \Gamma_n$, define 
\begin{equation}\label{TildePiDefinition}
\tilde{\Pi}(H) = \left\{\begin{array}{cl} \{2\}\cup\Pi(H) &
\text{ if $n\le 4$, $2\not\in\Pi(H)$, 
and } \delta_H(4q)>\delta_H(q)\\
\Pi(H) &\text{ otherwise}
\end{array}\right.
\end{equation}
where $q$ is the product of all odd primes in $\Pi(H)$.
Combining Theorems~\ref{SurjectiveImpliesCoprime} and~\ref{evenMtest} 
yields the next theorem.
\begin{theorem}
\label{BothPisAreClose}
If $H$ is arithmetic of level $M>1$ then $\pi(M)=\tilde{\Pi}(H)$.
\end{theorem}

We leave the problem of finding $\Pi(H)$ aside for the moment,
returning to it in Section~\ref{Section2}.

\subsubsection{}  \label{algosection}
Now we aim for our promised algorithm to compute $M$ from $\pi(M)$. 

\begin{lemma}\label{deltaGen} \ 
\begin{itemize}
\item[{\rm (i)}]
Suppose that $\delta_H(kp^{a})=\delta_H(kp^{a+1})$ for some prime 
$p$, positive integer $a$, and $k$ coprime to $p$.
Then $\delta_H(kp^{b})=\delta_H(kp^{a})$ for all $b\geq a$.
\item[{\rm (ii)}]  Let $p$, $a$, and $k$ be as in {\rm (i)}. 
Then $\delta_H(lp^b)=\delta_H(lp^{a})$ for all $b\geq a$ and any
multiple $l$ of $k$ such that $\pi(l)=\pi(k)$.
\end{itemize}
\end{lemma}
\begin{proof} 
(i) If $b>a+1$ is minimal subject to
$\delta_H(kp^b)\ne \delta_H(kp^{a})$
then $\delta_H(kp^{b-2})=\delta_H(kp^{b-1})=\delta_H(kp^a)$. 
So $\Gamma_{n,kp^{b-2}} \leq \Gamma_{n,kp^{b-1}}H$, implying 
that $\Gamma_{n,kp^{b-2}}\cap H$ is a proper supplement
of  $\Gamma_{n,kp^{b-1}}$ in $\Gamma_{n,kp^{b-2}}$. Since
\[
\Gamma_{n,kp^{b-2}}/\Gamma_{n,kp^b} 
\cong \Gamma_{n,p^{b-2}}/\Gamma_{n,p^b},
\]
with $\Gamma_{n,kp^{b-1}}/\Gamma_{n,kp^b}$ corresponding to 
$\Gamma_{n,p^{b-1}}/\Gamma_{n,p^b}$ under the isomorphism, this 
contradicts Lemma~\ref{lemmanosupplement}.

(ii) Suppose that there are $b\geq a$, and $l$ divisible by $k$ with 
$\pi(l) = \pi(k)$, such that $\delta_H(lp^{b})\ne \delta_H(lp^{b+1})$. 
By (i), $\delta_H(kp^{b})=\delta_H(kp^{b+1})$. 
Define $\bar{H} = \abk \Gamma_{n,kp^b} \cap \abk \Gamma_{n,lp^{b+1}}H$. 
We observe that $\Gamma_{n,kp^b} = \Gamma_{n,kp^{b+1}} \bar{H}$
and $\bar{H}/\Gamma_{n,lp^{b+1}}$ is a proper subgroup of
\[
\Gamma_{n,kp^{b}}/\Gamma_{n,lp^{b+1}}= \Gamma_{n,lp^{b}}/\Gamma_{n,lp^{b+1}}
\times \Gamma_{n,kp^{b+1}}/\Gamma_{n,lp^{b+1}}.
\]
The factors of this direct product have coprime orders: one is isomorphic 
to the $p$-group $\Gamma_{n,p^{b}}/\Gamma_{n,p^{b+1}}$; the other
is isomorphic to $\Gamma_{n,k}/\Gamma_{n,l}$, which is a $p'$-group
because $\pi(l)=\pi(k)$. Hence 
$\Gamma_{n,kp^{b+1}} \bar{H}= \Gamma_{n,kp^{b+1}} K$
for some $K < \Gamma_{n,lp^{b}}$ in $\bar{H}$ such that
$K\cap \Gamma_{n,kp^{b+1}}=\Gamma_{n,lp^{b+1}}$.
But then $\Gamma_{n,kp^{b+1}} \bar{H}\neq \Gamma_{n,kp^b}$.
\end{proof}

The following procedure computes the level of an arithmetic group $H$. 
The idea is to add higher powers of prime divisors of the level while 
$\delta_H$ increases, until $\delta_H$ reaches a stabilized value as 
dictated by Lemma~\ref{deltaGen}.
(We keep the specification of input and output completely
general at this stage.)

\vspace{16pt}

${\tt LevelMaxPCS}(S, \Omega)$

\vspace{7pt}

Input: a generating set $S$ for a subgroup $H\leq \Gamma_n$; 
a finite set $\Omega$ of primes.

Output: an integer $N$.

\vspace{2pt}

\begin{itemize}
\item[] For each $p\in\Omega$ let $\mu_p=1$ and 
$z_p=\prod_{q\in\Omega, q\neq p} q$.

\item[] While $\exists \, p\in\Omega$ such that
$\delta_H(p^{\mu_p+1} \cdot z_p)>\delta_H(p^{\mu_p} \cdot z_p)$
\begin{itemize} 
\item[] increment $\mu_p$ by $1$ and repeat.
\end{itemize}
\item[] Return  $N=\prod_{p\in\Omega}p^{\mu_p}$.
\end{itemize}

\vspace{5pt}

\begin{remark} 
The test for even $M$ in Theorem~\ref{evenMtest} (which is 
invoked only when $n\le 4$ and we have discovered that 
$2\not \in \Pi(H)$) makes a similar comparison of indices 
$\delta_H$, and can be implemented using the same subroutines 
as above.
\end{remark} 

\begin{remark} 
A reader might ask whether ${\tt LevelMaxPCS}$ is unduly complicated:
perhaps the least $p^a$ such that $\delta_H(p^a)=\delta_H(p^{a+1})$ is 
the $p$-part of the level of an input arithmetic group? This supposition 
is false, as the following example (constructed from a subdirect product 
of $\Gamma_{3,3}/\Gamma_{3,9}\cong C_3^8$ with a subgroup of 
$\PSL(3,5)$ of order $3$) illustrates. Let
\begin{eqnarray*}
&H=\Big \langle \, \Gamma_{3,45},
{\footnotesize
\left(\begin{array}{rrr}
1&30&0\\ 
0&1&0\\ 
0&0&1\\ 
\end{array}\right)}, 
{\footnotesize
\left(\begin{array}{rrr} 
-29&0&-30\\ 
0&1&0\\ 
30&0&31\\ 
\end{array}\right)}, 
{\footnotesize
\left(\begin{array}{rrr} 
-29&-45&15\\ 
30&1&30\\ 
30&0&31\\ 
\end{array}\right)}, 
\\
&\phantom{H=}
\hspace{20pt} {\footnotesize
\left(\begin{array}{rrr} 
1&0&0\\ 
15&-29&-30\\ 
30&30&31\\ 
\end{array}\right)}, 
{\footnotesize
\left(\begin{array}{rrr} 
16&15&0\\ 
-255&-239&0\\ 
0&0&1\\ 
\end{array}\right)}, 
{\footnotesize
\left(\begin{array}{rrr} 
16&15&30\\ 
-255&-239&15\\ 
0&0&1\\ 
\end{array}\right)}, 
\\
&\phantom{H=}
\hspace{10pt} 
{\footnotesize
\left(\begin{array}{rrr} 
1&0&30\\ 
0&1&30\\ 
0&0&1\\ 
\end{array}\right)}, 
{\footnotesize
\left(\begin{array}{rrr} 
10&0&9\\ 
36&-137&66\\ 
-99&-453&22\\ 
\end{array}\right)} 
\Big \rangle .
\end{eqnarray*}
Then $\delta_H(3)=\delta_H(9)=5616$, $\delta_H(5)=\delta_H(25)=124000$, 
$\delta_H(15)=696384000$, and $\delta_H(45)=2089152000$. Hence 
$H$ has level $45$, not $15$.

This example is a phenomenon of mixed primes: if the level of 
$H$ is a power $p^r$ for $p$ prime and unexceptional $(n,p)$, and
$\delta_H(p^a)=\delta_H(p^{a+1})$ for $a\geq 1$, then $r\leq a$. 
\end{remark}

The next theorem justifies correctness of ${\tt LevelMaxPCS}$ in our 
main situation.
\begin{theorem}\label{LevelMaxPCSJustified}
If $H=\langle S\rangle$ is arithmetic of level $M$ then 
${\tt LevelMaxPCS}$ with input $S$ and $\Omega = \pi(M)$ terminates, 
returning $M$. 
\end{theorem}
\begin{proof}
The values of $\delta_H$ encountered in the while-loop are bounded,
because $\delta_H(m)$ divides $\delta_H(M)$ for all $m$. 
Thus ${\tt LevelMaxPCS}$ terminates. 

If $p^a$ is the $p$-part of $M$, and $q \hspace{1.4pt} | \hspace{1.4pt} M$ 
is coprime to $p$, then $\delta_H(p^{a+1}q) = \delta_H(p^aq)$.
So the output $N$ of ${\tt LevelMaxPCS}$ must divide $M$.  
Since $\pi(N) = \pi(M)$, this implies that $\delta_H(M)= \delta_H(N)$
by Lemma~\ref{deltaGen}. Then $M\hspace{1.4pt} | \hspace{1.4pt} N$ by 
Lemma~\ref{BasicDeltaProperties}. 
\end{proof}

\section{Computing with Zariski dense subgroups}
\label{Section2}

Let $n> 2$ and $H$ be a finitely generated subgroup of 
$\Gamma_n= \SL(n,\Z)$ or $\Sp(n,\Z)$. We describe how to compute 
$\Pi(H)$ when $H$ is arithmetic, or, more generally, dense.
This relies on knowing a transvection $t \in H$ (which is true 
of all groups in Section~\ref{Section3}); and we restrict to odd
degree $n$ for $\G=\SL(n,\Z)$. Note that if $H$ is 
arithmetic then it contains transvections, whereas if $H$ is dense 
then it need not even contain a unipotent 
element~\cite[Proposition~5.3]{Venky87}. 
 
We also provide a simple algorithm to test density of $H$. Here 
again $H$ should contain a known transvection, and $n$ is odd if 
$\G=\SL(n,\Z)$.
Less restricted density testing algorithms are discussed in 
Subsection~\ref{Subsection22}. Then 
Subsection~\ref{ComputingPiForDenseInput} 
extends $\tt LevelMaxPCS$ to dense input groups.

\subsection{Density and transvections}\label{Subsection21}

We formulate various conditions for density. The first result is truly 
fundamental (see \cite{Lubotzky97}, 
\cite{Weisfeiler}, and \cite[Theorem~2.4]{Rivin}).
\begin{theorem}\label{Corollary2}
$H$ is dense if and only if $\varphi_p(H) = \varphi_p(\G)$ for 
some prime $p > 3$.
\end{theorem}

Let $\F$ be a field. An element $t$ of $\GL(n,\F)$ is a
\textit{transvection} if it is unipotent and $1_n-t$ has rank $1$.
\begin{theorem}[\cite{ZalesSeregkin}] 
\label{ZalesSeregkin}
Let $n> 2$ and $p$ be an odd prime. 
If $G\leq \GL(n,p)$ is irreducible and generated by transvections, 
then either $G=\SL(n, p)$ or $G$ is conjugate to $\Sp(n, p)$.
\end{theorem}

\begin{corollary}\label{Corollary4}
Suppose that $n> 2$, $p$ is an odd prime, and $G\leq \GL(n, p)$ 
has a transvection $t$ such that the normal closure 
$\langle t \rangle^G$ is irreducible.
Then $G$ contains $\SL(n, p)$ or a conjugate of $\Sp(n, p)$. In
particular,  $\SL(n, p)\leq G$ for odd $n$.
\end{corollary}

\begin{lemma}\label{Lemma5}
Suppose that $G$ is an irreducible subgroup of $\GL(n, \F)$ and
$t\in G$ is a transvection such that $\langle t \rangle^G$ is
reducible. Then $G$ is imprimitive.
\end{lemma}
\begin{proof}
By Clifford's Theorem, $\F^n =W_1\oplus \cdots \oplus W_k$ where 
$k>1$ and the $W_i$ are irreducible modules for $\langle t \rangle^G$.
Then $t|_{W_i}$ must be a transvection for some $i$, 
and $t|_{W_j}= 1_{W_j}$ for $j \neq \abk i$. Thus
$\F^n$ has more than one homogeneous component. 
\end{proof}

\begin{corollary}\label{Corollary6}
If $H\leq \G$ has a transvection $t$ such that $\langle t \rangle^H$
is not absolutely irreducible, then $H$ is not dense.
\end{corollary}
\begin{proof}
We may assume that $H$ is absolutely irreducible, so that $H$ is 
imprimitive by Lemma~\ref{Lemma5}. Since $\varphi_p(\G)$ is (absolutely) 
primitive, $\varphi_p(H) \neq\varphi_p(\G)$ for almost all primes $p$.
\end{proof}

\begin{corollary}\label{Corollary7}
Let $G \leq \GL(n, \F)$ and $t\in G$ be a transvection. Then
$\langle t \rangle^G$ is irreducible if and only if $G$ is primitive.
\end{corollary}
\begin{proof}
One direction follows from \cite[(1.9)]{ZalesSeregkin}, the other
from Lemma~\ref{Lemma5}.
\end{proof}

In Proposition~\ref{Proposition8} and Lemma~\ref{Lemma9}, 
$H\leq \SL(n, \Z)$ for odd $n> 2$, or $H\leq \Sp(n, \Z)$ for $n>2$. 
\begin{proposition}\label{Proposition8} 
Suppose that $H$ contains a transvection $t$. Then $H$ is dense if and only 
if $\langle t \rangle^H$ is absolutely irreducible.
\end{proposition}
\begin{proof}
Put $N=\langle t \rangle^H$.
If $H$ is dense then it is absolutely irreducible and primitive;
so $N$ is absolutely irreducible by Corollary~\ref{Corollary7}.

Suppose that $N$ is absolutely irreducible. Since there exists an
odd prime $p$ such that $\varphi_p(N)$ is absolutely irreducible, and 
$\varphi_p(N)$ contains the transvection $\varphi_p(t)$, 
Theorem~\ref{Corollary2} and Corollary~\ref{Corollary4} imply that 
$H$ is dense.
\end{proof}

\begin{lemma}\label{Lemma9}
$H$ is dense if and only if there are
a prime $p>3$ and a transvection $t\in\abk \varphi_p(H)$ such that 
$\langle t \rangle^{\varphi_p(H)}$ is irreducible.
\end{lemma}
\begin{proof}
If $H$ is dense then $\varphi_p(H) = \SL(n,p)$ or $\Sp(n,p)$
for a prime $p>3$. Therefore $\varphi_p(H)$ contains a transvection 
$t$, and $\langle t \rangle^{\varphi_p(H)}= \varphi_p(H)$ is 
irreducible. The converse follows from Theorem~\ref{Corollary2} 
and Corollary~\ref{Corollary4}.
\end{proof}

\begin{lemma}\label{Lemma10}
Suppose that $n > 2$ is prime and $H$ is an absolutely irreducible
subgroup of $\SL(n,\Z)$ containing a transvection. Then $H$ is dense.
\end{lemma}
\begin{proof}
Let $t \in H$ be a transvection. If $\langle t \rangle^H$ is not
absolutely irreducible then it is monomial. But $t$ is certainly not
monomial. Proposition~\ref{Proposition8} gives the result.
\end{proof}

\subsection{Algorithms to test density and compute $\Pi$}
\label{Subsection22} 
 
Assume that $n$ is odd if $\Gamma_n=\SL(n,\Z)$, and we know a 
transvection $t$ in $H\leq \G$. By Proposition~\ref{Proposition8}, 
testing density of $H$ is the same as testing absolute irreducibility
of $N = \langle t \rangle^H$. The latter may be carried out using the 
procedure ${\tt BasisAlgebraClosure}$ in \cite[p.~401]{Tits}.
This returns a basis of the enveloping algebra $\langle N\rangle_{\Q}$,
as words over an input generating set $S$ for $H$.
So we have the following density testing algorithm.

\vspace{16pt}

${\tt IsDense}(S, t)$

\vspace{7pt}

Input: a finite subset $S$ of $\Gamma_n$ and a transvection 
$t\in H=\langle S\rangle$.

Output: $\tt true$ if $H$ is dense; $\tt false$ otherwise.

\begin{itemize}
\item[] $\mathcal{A} = {\tt BasisAlgebraClosure}(t,S)$.

\item[] Return $\tt true$ if $|\mathcal{A}|=n^2$; else 
return $\tt false$.
\end{itemize} 

\vspace{3pt}

\begin{remark}

(i) When $n$ is prime, it suffices to test whether $H$ itself is 
absolutely irreducible (Lemma~\ref{Lemma10}).

(ii)\, 
By Corollary~\ref{Corollary6}, if $\langle t \rangle^H$ is 
not absolutely irreducible then $H$ is not dense. 
If $\langle t \rangle^H$ is absolutely irreducible and $n$ is even then
$\varphi_p(H)$ could be conjugate to $\Sp(n, p)$, so
we must proceed by other means to decide whether $H$ is dense.
\end{remark}

We now discuss computing $\Pi$ for dense $H=\langle S
\rangle \leq \Gamma_{n}$, given a transvection $t\in \abk H$. 
Let $\mathcal{A}=\{A_1, \ldots, A_{n^2}\}\subseteq H$ be a basis of 
$\langle\langle t\rangle^H \rangle_{\Q}$.
Form the matrix $[\mathrm{tr}(A_iA_j)]_{ij}$
and denote its determinant by $d$. Let $\Pi_1$ be the set consisting 
of $\pi(d)$ together with the prime divisors of all non-zero non-diagonal 
entries of $t$. Then $\Pi\subseteq \Pi_1$. To obtain $\Pi$, we check whether 
$\varphi_p(H) = \varphi_p(\Gamma_{n})$ for $p$ running over $\Pi_1$. 
Call this process ${\tt PrimesForDense}(S, t)$.

If we have an upper bound on the primes in $\Pi(H)$ then 
we can find $\Pi(H)$. Such a bound may be derived from 
\cite[pp.~10--11]{Breuillard} (a quantitative version of strong 
approximation). Alternatively, we could use a Hadamard-type 
inequality for the matrix determinant associated to a basis 
$\mathcal{A}$ as above. 
However, the bounds resulting from either approach are too large to be
practical.
 
Our algorithms in this subsection need an input transvection. As
noted, a dense group may not contain unipotent elements. 
Moreover, unipotent elements are `rare'~\cite{LubotzkyMeiri}. 
We make some brief remarks about density testing 
(in any degree $n> 2$) without this constraint.

A dense group is absolutely irreducible and not solvable-by-finite. 
Both of these properties can be readily tested~\cite{Tits}, which serves 
as a preliminary check. Note that if $H$ is absolutely irreducible and 
contains a non-trivial unipotent element (e.g., a transvection) then $H$ 
is not solvable-by-finite.

Monte Carlo and deterministic algorithms for density testing 
are given in \cite{Rivin1}. In Section~\ref{Miscellaneous}, we compare 
our implementations of these algorithms and $\tt IsDense$.
Further afield, see \cite{Derksen} for an algorithm to compute Zariski 
closures, which could be applied to test density. 

\subsection{Computing the minimal arithmetic overgroup}
\label{ComputingPiForDenseInput}

Let $n> 2$ and $H= \langle S\rangle<\Gamma_n$ be dense. As 
Martin Kassabov has pointed out, there are only finitely 
many arithmetic subgroups of $\Gamma_n$ containing 
$H$~\cite{kassabovcommunication}. Their intersection is the 
\emph{minimal} arithmetic overgroup of $H$.

 We generalize 
Theorems~\ref{BothPisAreClose} and \ref{LevelMaxPCSJustified}, 
thereby proving that $\tt LevelMaxPCS$ terminates for 
input $H$, returning the level of its minimal arithmetic 
overgroup. 
\begin{lemma}\label{lIsaPiNumber} 
If $\, l$ is the level of the minimal arithmetic overgroup of
$H$ then $\pi(l)=\tilde{\Pi}(H)$ as defined in 
\eqref{TildePiDefinition}.
\end{lemma}
\begin{proof} 
Let $C$ be the overgroup. For any $m$ we have 
$\Gamma_{n,m}H = \Gamma_{n,m}C$, because $C=\Gamma_{n,l}H$ is 
contained in the arithmetic group $\Gamma_{n,m}H$. Thus 
$\Pi(C)= \Pi(H)$, and  $\delta_H(m)=\delta_H(m')$ 
if and only if $\delta_C(m)=\delta_C(m')$. The assertion 
is now evident from Theorem~\ref{BothPisAreClose}.
\end{proof}

\begin{theorem}\label{LevelMaxPCSForDense}
${\tt LevelMaxPCS}$ with input $S$ and $\tilde{\Pi}(H)$ 
terminates, returning the level of the minimal arithmetic 
overgroup $C$ of $H$. 
\end{theorem}
\begin{proof}
Since $\delta_H(m)\leq |\Gamma_n:C|$ for all $m$, 
${\tt LevelMaxPCS}$ terminates. 
Say the output is $M$. By Lemma~\ref{lIsaPiNumber}, $\pi(M) = \pi(l)$ 
and $M$ divides $l$. 
Then $\delta_C(M) =\abk\delta_C(l)$ by Lemma~\ref{deltaGen}, so 
$l$ divides $M$. 
\end{proof}

\section{Experimental results}\label{Section3}

We implemented the algorithms of Sections~\ref{Section1} and
\ref{Section2} in {\sf GAP}, relying on the packages 
 `matgrp'~\cite{gapmatgrp} and `recog'~\cite{gaprecog}. 
In this section we describe how we used our implementation to solve 
problems for important classes of groups that have been
the focus of much activity. 
Times quoted for all experiments are in seconds, on a 3.7 GHz Quad-Core 
late 2013 Mac Pro with 32 GB memory.

\subsection{Small subgroups of $\SL(3, \Z)$}\label{Subsubsection231}

Lubotzky~\cite{Lubotzky86} asked whether every arithmetic 
subgroup of $\G=\SL(n, \Z)$ for $n > 2$ has a $2$-generator subgroup 
of finite index. To support an affirmative answer to this question,
the following groups were studied in \cite[p.~414]{LongReidI}. Let
$G = \langle x, y, z\; | \; z  x z^{-1} = x y, \, z y z^{-1} = y x
 y \rangle$ and $F = \langle x, y \rangle\leq G$.
For $T\in \Z$ define the homomorphism $\beta_T : G \rightarrow \SL(3,\Z)$ 
by 

\vspace{-10pt}

\[
x \mapsto X_T = {
\left(\begin{array}{ccc} -1+T^3 & -T &
T^2
\\
0 & -1 & 2T
\\
-T & 0 & 1 \end{array} \right)}, \ \ \ y\mapsto Y_T = {
\left(\begin{array}{ccc} -1 & 0 & 0
\\
-T^2 & 1 & -T
\\
T & 0 & -1 \end{array} \right)},
\]
\[
z\mapsto Z_T = {
\left(\begin{array}{ccc} 0 & 0 & 1
\\
1 & 0 & T^2
\\
0 & 1 & 0 \end{array} \right)}.
\]

\begin{lemma} {\rm(Cf.~\cite[Theorem 2.6]{LongReidI}.)}
If $T \neq 0$ then $\beta_T(F)$ is dense.
\end{lemma}
\begin{proof}
The element $b_1 = X_T^{-1}Y_T^{3}X_TY_T^{2}X_TY_T^{-1}X_T$
is a transvection~\cite[p.~418]{LongReidI}. 
As $\beta_T(F)$ is absolutely irreducible, the result
follows from Lemma~\ref{Lemma10}.
\end{proof}

\begin{theorem}{\rm (\cite[Theorem 3.1]{LongReidI}.)}
If $T \neq 0$ then $\beta_T(F)$ is arithmetic.
\end{theorem}

Earlier attempts to compute $|\Gamma_3:\beta_T(F)|$ 
failed~\cite[pp.~419, 423]{LongReidI}. 
We compute these indices by first determining $\pi(M)$ from
${\tt PrimesForDense}(\{X_T, Y_T\}, b_1)$ via
Theorem~\ref{BothPisAreClose}.
Then $M={\tt LevelMaxPCS}(\{X_T, Y_T\}, \abk \pi(M))$.
Table~1 displays sample results.

\begin{table}[htb]
\begin{tabular}{|r|r|r|r|}
\multicolumn{1}{c}{$T$} & \multicolumn{1}{c}{$M$} 
& \multicolumn{1}{c}{Index} &\multicolumn{1}{c}{Time}
\\
\hline
$-1$
&$11$
&$7
{\cdot}19$
&6.1
\\
$-2$
&$2^{6}$
&$2^{19}7$
&6.7
\\
$1$
&$5$
&$31$
&5.6
\\
$2$
&$2^{5}$
&$2^{17}7$
&6
\\
$3$
&$3^{3}73$
&$2^{3}3^{11}13
{\cdot}1801$
&29.2
\\
$4$
&$2^{7}23$
&$2^{31}7^{2}79$
&16
\\
$5$
&$5^{3}367$
&$2^{4}3^{2}5^{10}13
{\cdot}31
{\cdot}3463$
&143.8
\\
$6$
&$2^{8}3^{3}5$
&$2^{29}3^{10}7
{\cdot}13
{\cdot}31$
&26.5
\\
$7$
&$7^{3}1021$
&$2^{5}3^{4}5
{\cdot}7^{10}19
{\cdot}347821$
&570.7
\\
$8$
&$2^{10}191$
&$2^{46}7^{2}13^{2}31$
&98.6
\\
$9$
&$3^{6}2179$
&$2^{3}3^{27}7
{\cdot}13
{\cdot}226201$
&1652.1
\\
$10$
&$2^{5}5^{3}11
{\cdot}17$
&$2^{26}3
{\cdot}5^{10}7^{2}19
{\cdot}31
{\cdot}307$
&50.7
\\
$11$
&$5
{\cdot}11^{3}797$
&$2^{4}5^{2}7
{\cdot}11^{10}19
{\cdot}31
{\cdot}157
{\cdot}4051$
&1344.6
\\
$12$
&$2^{7}3^{3}647$
&$2^{35}3^{10}7
{\cdot}13
{\cdot}211
{\cdot}1987$
&721.4
\\
$13$
&$13^{3}29
{\cdot}227$
&$2^{4}3^{2}7
{\cdot}13^{11}61
{\cdot}67
{\cdot}73
{\cdot}709$
&246
\\
$14$
&$2^{6}7^{3}257$
&$2^{28}3^{3}7^{11}19
{\cdot}61
{\cdot}1087$
&195.5
\\
$15$
&$3^{3}5^{3}67
{\cdot}151$
&$2^{9}3^{14}5^{10}7^{3}13
{\cdot}31^{2}1093$
&272.5
\\
$16$
&$2^{13}5
{\cdot}307$
&$2^{63}3^{3}7
{\cdot}31
{\cdot}43
{\cdot}733$
&259.3
\\
$18$
&$2^{5}3^{6}1093$
&$2^{23}3^{27}7
{\cdot}13^{2}398581$
&844
\\
$19$
&$19^{3}67
{\cdot}307$
&$2^{4}3^{9}5
{\cdot}7^{2}19^{10}31
{\cdot}43
{\cdot}127
{\cdot}733$
&466.6
\\
$20$
&$2^{7}5^{3}2999$
&$2^{36}3
{\cdot}5^{10}7
{\cdot}13
{\cdot}31
{\cdot}613
{\cdot}1129$
&13309.4
\\
$50$
&$2^{5}5^{6}23
{\cdot}1019$
&$2^{24}3
{\cdot}5^{25}7^{3}31
{\cdot}79
{\cdot}148483$
&2584.7
\\
$100$
&$2^{7}5^{6}29
{\cdot}67
{\cdot}193$
&$2^{42}3^{5}5^{25}7^{4}13
{\cdot}31^{2}67
{\cdot}1783$
&892.6
\\
\hline
\end{tabular}

\bigskip

\caption{}
\label{Table1}
\end{table}

\begin{remark}
Lubotzky's question has been answered affirmatively~\cite{Meiri}.
\end{remark}
 
Another representation $\rho_k: G\rightarrow \Gamma_3$
is defined in \cite[p.~414]{LongReidI} by
\[
\rho_k(x) =  
{
\left(\begin{array}{ccc} 1 & -2 & 3
\\
0 & k & -1-2k
\\
0 & 1 & -2 \end{array} \right)}, \quad \rho_k(y)  
= {
\left(\begin{array}{ccc} -2-k & -1 & 1
\\
-2-k & -2 & 3
\\
-1 & -1 & 2 \end{array} \right),}
\]
\[
\rho_k(z) 
= {
\left(\begin{array}{ccc} 0 & 0 & 1
\\
1 & 0 & -k
\\
0 & 1 & -1-k \end{array} \right)}.
\]

\vspace{1pt}

\noindent If $k\in\{ 0, 2, 3, 4, 5\}$ then $\rho_k(F)$ is arithmetic. 
Since a transvection in each group is known~\cite[p.~419]{LongReidI},
as before we can use $\tt{PrimesForDense}$ to 
find $\Pi$, then $\tt{LevelMaxPCS}$ with 
Theorem~\ref{BothPisAreClose} to compute levels. 
The $\rho_k(F)$ and $\rho_k(G)$ relate to open 
conjectures in \cite[Section~5]{LongReidI}. Table~\ref{Table2} 
solves the main problem, namely finding indices in $\Gamma_3$. 
The last column states the time to compute $|\Gamma_3:\rho_k(G)|$. 

\begin{table}[htb]

\begin{tabular}{|r|r|r|r|r|}
\multicolumn{1}{c}{$k$}
&\multicolumn{1}{c}{Level}
&\multicolumn{1}{c}{$|\Gamma_3 : \rho_k(G)|$}
&\multicolumn{1}{c}{$|\Gamma_3 : \rho_k(F)|$}
&\multicolumn{1}{c}{Time}
\\
\hline
\ $0$ \
&$11$
&$7
{\cdot}19$
& $2
{\cdot}5
{\cdot}7
{\cdot}19$
&6
\\
\ $2$ \
&$2^{2}5
{\cdot}7$
&$2^{12}3^{2}5
{\cdot}7^{2}19
{\cdot}31$
&$2^{12}3^{3}5
{\cdot}7^{2}19
{\cdot}31$
&15.4
\\
\ $3$ \
&$13$
&$2^{2}3
{\cdot}13^{2}61$
&$2^{3}3^{2}13^{2}61$
&7
\\
\ $4$ \
&$3^{3}7$
&$2^{4}3^{11}7^{2}13
{\cdot}19$
&$2^{6}3^{13}7^{2}13
{\cdot}19$
&11.3
\\
\ $5$ \
&$2^{2}19
{\cdot}31$
&$2^{10}3^{3}5
{\cdot}31^{2}127
{\cdot}331$
&$2^{11}3^{5}5
{\cdot}31^{2}127
{\cdot}331$
&49.1
\\
\hline
\end{tabular}

\bigskip

\caption{}
\label{Table2}
\end{table}

\subsection{Monodromy groups}
\label{Subsubsection232}
Let 
$f({\rm x}) = \prod_{j = 1}^{n} 
({\rm x} - a_j) = 
{\rm x}^n + A_{n-1}{\rm x}^{n-1} + \cdots +\abk A_0
$
and
$
g({\rm x}) = \prod_{j = 1}^{n} 
({\rm x} - b_j) = {\rm x}^n + B_{n-1}{\rm x}^{n-1} + \cdots + B_0
$
where $a_j = e^{2\pi{\rm i}\alpha_j}$ and $b_j = e^{2\pi{\rm i}\beta_j}$ 
for $\alpha_j$, $\beta_j\in \mathbb{C}$, $1 \leq j\leq n$.
The group $H$ generated by the companion matrices
\[
A =  {
\left(\begin{array}{cccc} 0 & \cdots & 0 & -A_0
\\
1 & \cdots & 0 & -A_1
\\
\vdots & \ddots & \vdots & \vdots
\\
0 & \cdots & 1 & -A_{n-1}
\end{array} \right)}, \quad
B =  {
\left(\begin{array}{cccc} 0 & \cdots & 0 & -B_0
\\
1 & \cdots & 0 & -B_1
\\
\vdots & \ddots & \vdots & \vdots
\\
0 & \cdots & 1 & -B_{n-1}
\end{array} \right)}
\]
of $f({\rm x})$ and $g({\rm x})$ is the
\textit{hypergeometric group} corresponding to $f({\rm x})$
and $g({\rm x})$. It is the monodromy group of a
hypergeometric ordinary differential equation (see 
\cite[pp.~331--332]{BH}, \cite[p.~334]{BravThomas},
\cite[p.~592]{SinghVenky}). 

Suppose that $f({\rm x})$, $g({\rm x})\in \Z[{\rm x}]$ are reciprocal 
($f({\rm x})= {\rm x}^nf(1/{\rm x})$ and 
$g({\rm x})= {\rm x}^ng(1/{\rm x})$) with no common roots in 
$\C$. Further suppose that they constitute a primitive pair (there do 
not exist $f_1({\rm x})$, $g_1({\rm x}) \in \Z[{\rm x}]$ and $k\geq 2$ 
such that $f({\rm x}) = f_1({\rm x}^k)$ and 
$g({\rm x}) = g_1({\rm x}^k)$). Then $H \leq\Sp(\Phi, \Z)$ for some 
non-degenerate integral symplectic form $\Phi$ on 
$\Z^n$~\cite[p.~592]{SinghVenky}.

There are $14$ pairs $(f({\rm x}), g({\rm x}))$ with 
$g({\rm x}) \in \Z[{\rm x}]$ coprime to  $f({\rm x})=({\rm x}-1)^4$ 
such that the roots of $g({\rm x})$ are roots of 
unity~\cite[pp.~595, 615]{SinghVenky}. 
The group $H=\langle A, B \rangle$ in these cases
is a monodromy group associated with Calabi--Yau threefolds. 
Seven such $H$ are arithmetic, and the 
rest are thin~\cite{BravThomas,Singh,SinghVenky}.
In \cite[p.~175]{MonodromyAppendix}, $H$ is shown to be 
$\GL(4,\Q)$-conjugate to $G(d, k):= \langle U, T \rangle 
\leq \Sp(4, \Z)$ where 
\[
U =  
{
\left(\begin{array}{crrr} 1 & 1 & 0 & \ 0
\\
0 & 1 & 0 & \ 0
\\
d & d & 1 & \ 0
\\
0 & -k & -1 & \ 1
\end{array} \right)},
\quad  T =  
{
\left(\begin{array}{cccc} 
1 & 0 & 0 & 0
\\
0 & 1 & 0 & 1
\\
0 & 0 & 1 & 0
\\
0 & 0 & 0 & 1
\end{array} \right)}.
\]
Note that conjugation preserves
arithmeticity~\cite[p.~87]{Humphreys}. For $d_2 \, | \, d_1$, let  
$\hat{G}(d_1,d_2)$ be the subgroup of $\Sp(4,\Z)$ consisting of all $h$ 
satisfying
\begin{align*}
h \equiv
{
\renewcommand{\arraycolsep}{.2cm}\left(\begin{array}{cccc}
1 & * & * & * \\
0 & * & * & * \\
0 & 0 & 1 & 0 \\
0 & * & * & *
 \end{array} \right) } \hspace{-5pt} \mod d_1 \quad \text{and} \quad 
h \equiv
{
\renewcommand{\arraycolsep}{.2cm}\left(\begin{array}{cccc}
1 & * & * & * \\
0 & 1 & * & * \\
0 & 0 & 1 & 0 \\
0 & 0 & * & 1
 \end{array} \right)} \hspace{-5pt} \mod d_2 .
\end{align*}
If $d_1=d$ and $d_2= \mathrm{gcd}(d,k)$ then $\hat{G}(d_1, d_2)$ 
is an arithmetic subgroup of $\Sp(4, \Z)$ containing $G(d, k)$. 
By \cite[Appendix]{MonodromyAppendix},
\begin{equation}
\label{HatGinSpIndex}
|\Sp(4, \Z) : \hat{G}(d_1, d_2)| = d_1^4 \cdot 
{\textstyle \prod}_{p\mid d_1} (1 - p^{-4}) \cdot
d_2^2 \cdot {\textstyle \prod}_{p\mid d_2}(1 - p^{-2}).
\end{equation}
The overgroup $\hat{G}(d_1, d_2)$ could be used to investigate properties 
of $G(d, k)$, such as bounds on $|\Sp(4, \Z) : G(d, k)|$; 
cf.~\cite[p.~6]{HofStrat}. Our implementation enables us to complete 
such tasks quickly, including those not completed in 
\cite[p.~6]{HofStrat}. Also, for the first time we can determine the 
minimal arithmetic overgroup of $G(d,k)$.

We compute $\Pi(G(d,k))$ via ${\tt PrimesForDense}$ with
input transvection $T$, then the level and index of $G(d,k)$
via ${\tt LevelMaxPCS}$. See Table~\ref{TableX}.

\begin{table}[htb]

\begin{tabular}{|c|c|c|c|r|c|c|}
\multicolumn{1}{c}{$(\alpha_1,\alpha_2)$}&
\multicolumn{1}{c}{$(d,k)$}
&\multicolumn{1}{c}{$M$}
&\multicolumn{1}{c}{Index $G$}
&\multicolumn{1}{c}{Time}
&\multicolumn{1}{c}{Index $\hat G$}
\\
\hline
$(\frac1{10},\frac3{10})$&
$(1,3)$
&$2$
&$6$
&5.8
&1
\\
$(\frac16,\frac16)$&
$(1,2)$
&$2$
&$10$
&5.4
&1
\\
$(\frac16,\frac14)$&
$(2,3)$
&$2^{3}$
&$2^{6}3
{\cdot}5$
&7.1
&$3\cdot 5$
\\
$(\frac16,\frac13)$&
$(3,4)$
&$2^{2}3^{2}$
&$2^{9}3^{5}5^{2}$
&12.6
&$2^4 5$
\\
$(\frac14,\frac14)$&
$(4,4)$
&$2^{6}$
&$2^{20}3^{2}5$
&10.1
&$2^6 3^25$
\\
$(\frac14,\frac13)$&
$(6,5)$
&$2^{3}3^{2}$
&$2^{10}3^{6}5^{2}$
&15.6
&$2^43{\cdot}5^2$
\\
$(\frac13,\frac13)$&
$(9,6)$
&$2
{\cdot}3^{5}$
&$2^{8}3^{14}5^{2}$
&19.2
&$2^7 3^4 5$
\\
$(\frac15,\frac25)$&
$(5,5)$
&$2
{\cdot}5^{3}$
&$2^{8}3^{3}5^{8}13$
&11.9
&$2^7 3^2 13$
\\
$(\frac18,\frac38)$&
$(2,4)$
&$2^{4}$
&$2^{11}3^{2}5$
&7.3
&$3^2 5$
\\
$(\frac1{12},\frac5{12})$&
$(1,4)$
&$2^{2}$
&$2^{5}5$
&5.8
&1
\\
$(\frac12,\frac12)$&
$(16,8)$
&$2^{10}$
&$2^{40}3^{2}5$
&20.1
&$2^{16} 3^2 5$
\\
$(\frac13,\frac12)$&
$(12,7)$
&$2^{5}3^{2}$
&$2^{17}3^{6}5^{2}$
&25
&$2^8 3{\cdot} 5^2$
\\
$(\frac14,\frac12)$&
$(8,6)$
&$2^{7}$
&$2^{24}3^{2}5$
&12.3
&$2^83^2 5$
\\
$(\frac16,\frac12)$&
$(4,5)$
&$2^{5}$
&$2^{13}3
{\cdot}5$
&9.9
&$2^4 3 \cdot 5$
\\
\hline
\end{tabular}

\bigskip

\caption{}
\label{TableX}
\end{table}

The arithmetic $G(d,k)$ appear in rows $1$--$7$. 
For $G(d,k)$ in any other row, we report the level and index 
of its minimal arithmetic overgroup. The first column  
defines $(\alpha_1, \alpha_2, \alpha_3, \alpha_4)$ for $A$, 
as $\alpha_3 = 1 - \alpha_2$ and $\alpha_4 = 1 -\alpha_1$.
`Time' is time to compute the level $M$, `Index $G$' is index 
of the minimal arithmetic overgroup in $\Sp(4, \Z)$, and 
`Index~$\hat{G}$' is $|\Sp(4, \Z) : \hat{G}(d_1, d_2)|$ from 
\eqref{HatGinSpIndex}.

\begin{remark} 
Table~\ref{TableX} shows that $\hat{G}(d_1, d_2)$
need not be the minimal arithmetic overgroup of $G(d, k)$. 
For instance, if $G(d, k)$ is arithmetic then it could differ
from $\hat{G}(d_1,d_2)$. Also note that arithmeticity
of groups of small index could in principle be determined by a 
coset enumeration once generators have been expressed as words 
in generators of $\Sp$.
\end{remark}

\section{Comparison of density testing algorithms}
\label{Miscellaneous}
We now compare our implementations of the density testing algorithms 
suggested in~\cite{Rivin1} and Subsection~\ref{Subsection22}.

$\mathtt{IsDenseIR1}$ is \cite[Algorithm~1]{Rivin1}. It accepts a 
finite subset $S$ of $\G = \SL(n,\Z)$ or $\mathrm{Sp}(n,\Z)$, $n> 2$, 
and tests whether $H = \langle S\rangle$ is dense.
This is a Monte Carlo algorithm based on random choice of elements in 
$H$ that have characteristic polynomial with large Galois group. Such 
elements are ubiquitous, in contrast to unipotent elements.
$\mathtt{IsDenseIR1}$ returns $\mathtt{true}$ if it detects 
non-commuting $g$, $h\in H$ such that $h$ has infinite order, and the 
Galois group of the characteristic polynomial of $g$ is $\mathrm{Sym}(n)$ 
if $\G = \SL(n,\Z)$ or $C_2 \wr \mathrm{Sym}(n/2)$ if 
$\G=\mathrm{Sp}(n,\Z)$. An output message $\mathtt{true}$ means that $H$ 
is dense, whereas $\mathtt{false}$ means that  
suitable $g$, $h$ were not found (in that event, $H$ may still be dense). 
We use an intrinsic {\sf GAP} function to compute Galois groups. Attempts 
at selecting random elements by the default method for finite groups, 
product replacement~\cite{Celleretal}, failed due to entry explosion in 
a precomputation step. So we took random words in the generators of 
length up to $50$. These elements might be of poor quality. 
Indeed, sometimes the algorithm as implemented did not
establish density. For $G_1$ below this happened about 40\% of the time.
The error rate could be reduced by a better choice of random elements,
or by an iteration over more random elements, but at the cost of runtime. 

The algorithm of~\cite[p.~23]{Rivin1}, which we call
$\mathtt{IsDenseIR2}$, is deterministic.
It accepts a finitely generated subgroup $H$ of a semisimple 
algebraic group $\calg(\mathbb{F})$, $\cha\,\mathbb{F} = 0$, and 
tests whether $H$ is finite and whether the adjoint representation of
$H$ in $\GL(m, \mathbb{F})$ is absolutely irreducible, where 
$m$ is the dimension of the Lie algebra of $\calg(\mathbb{F})$.
By incorporating methods from \cite{Recog}, this algorithm can be 
implemented over any field $\mathbb{F}$ of characteristic $0$. 

In Table~\ref{TableY}, $N$ is number of generators, and IR1, IR2, DFH 
are runtimes for our {\sf GAP} implementations of $\mathtt{IsDenseIR1}$, 
$\mathtt{IsDenseIR2}$, $\mathtt{IsDense}$, respectively. 

\begin{table}[htb]

{
\begin{tabular}{|c|c|c|c|c|c|c|}
\multicolumn{1}{c}{Group}
&\multicolumn{1}{c}{$n$}
&\multicolumn{1}{c}{$N$}
&\multicolumn{1}{c}{Output}
&\multicolumn{1}{c}{$\tt IR1$}
&\multicolumn{1}{c}{$\tt IR2$}
&\multicolumn{1}{c}{$\tt DFH$}
\\
\hline
$G_1$
&$5$
&$4$
&$\tt true$
&0.01
&11600
&0.2
\\
$G_2$
&$3$
&$3$
&$\tt true$
&0.02
&0.2
&0.04
\\
$G_3$
&$7$
&$48$
&$\tt true$
&0.05
&$-$
&5
\\
$G_4$
&$3$
&$3$
&$\tt true$
&0.01
&4.2
&0.2
\\
$G_5$
&$3$
&$3$
&$\tt true$
&0.01
&7.2
&0.3
\\
$G_6$
&$3$
&$3$
&$\tt true$
&0.01
&8.4
&0.2
\\
$G_7$
&$5$
&$15$
&$\tt false$
&0.01
&16200
&0.5
\\
$G_8$
&$5$
&$10$
&$\tt false$
&0.01
&24.4
&0.7
\\
$G_9$
&$11$
&$13$
&$\tt false$
&0.01
&$-$
&1.2
\\
\hline
\end{tabular}
}

\bigskip

\caption{}
\label{TableY}
\end{table}

The test groups $G_i$ were selected to vary $n$, $N$, and group structure. 
We know a transvection in each group (often as one of the generators).

$G_1$, $G_2$, $G_3$ are arithmetic. 
$G_1$ is $\SL(5,\Z)$, but not on the canonical generating set of
elementary matrices. The congruence image of
$G_2\leq \SL(3,\Z)$ is a $\{7,79\}$-Hall subgroup of 
$\SL(3, \Z_{2^3 23^2})$. It has level $2^3 23^2$ and index 
$2^24 3^2 11^2 23^{11}$.
$G_3\leq \SL(7,\Z)$ is generated by $E_{7,3^4 5 7^2}$ and 
the block diagonal matrices $\mathrm{diag}(h_1,h_2,1)$ where 
$h_1\in \beta_2(G)$ and $h_2\in \rho_4(G)$ for $G$, 
$\beta_T$, $\rho_k$ as in Subsection~\ref{Subsubsection231}.
It is arithmetic of level $3^85^27^4$.

$G_4$, $G_5$, $G_6$ are the groups generated by the transvections
\begin{eqnarray*}
&T_1 = {
\left(\begin{array}{ccc} 1 & x^2+1 & x
\\
0 & 1 & 0
\\
0 & 0 & 1 \end{array} \right)}, \qquad T_2 = {
\left(\begin{array}{ccc} 1 & 0 & 0
\\
x & 1 & x+1
\\
0 & 0 & 1 \end{array} \right)},
\\
&T_3 = {
\left(\begin{array}{ccc} 1 & 0 & 0
\\
0 & 1 & 0
\\
-x+1 & x^2 & 1 \end{array} \right)}
\end{eqnarray*}
for $x=11$, $99$, $998$ respectively. By \cite{SHumphries},
these groups are free and surject onto $\SL(3, p)$ modulo $p$ for all 
primes $p$ ($\tt PrimesForDense$ tells us that $\Pi(G_i) = \emptyset$ 
too); i.e., they are thin. As these $G_i$ also surject modulo $4$,  
they are congruent to $\SL(3, \Z_m)$ modulo $m$ for $m \geq 2$.

The last three groups are not dense.
$G_7$ is generated by $\mathrm{diag}(h_1,h_2)\in \SL(5,\Z)$ where 
$h_1$, $h_2$ are generators of $\beta_5(G)$, $\SL(2,\Z)$ respectively,
together with the upper triangular elementary matrices.
$G_8$ is the group of $5\times 5$ upper unitriangular matrices. 
$G_9$ is generated by $\mathrm{diag}(h_1,h_2)\in \SL(11,\Z)$ where 
$h_1$, $h_2$ range over generating sets for $\SL(6,\Z)$ and $\SL(5,\Z)$,
respectively, together with five randomly chosen upper unitriangular matrices.

More detail is available \href{http://www.math.colostate.edu/~hulpke/examples/densityex.g}{here}.

We write `$-$' in Table~\ref{TableY} if $\tt{IsDenseIR2}$ did not 
terminate within $12$ hours. This occurred for input of degree greater 
than $5$ (the adjoint representation leads to calculation in an 
$(n^2-1)^2$-dimensional lattice). 
Finally, remember that $\tt IsDense$ facilitates the computation of $\Pi(H)$
for dense input $H$, unlike $\tt IsDenseIR1$ and $\tt IsDenseIR2$. 

\vspace{7pt}

\noindent {\bf Acknowledgments.}
We are indebted to Professors Willem de Graaf, Derek Holt, Martin Kassabov, 
and T.~N.~Venkataramana for their vital assistance.

%

\bibliographystyle{amsplain}

\end{document}